\documentclass[11pt]{article}
\usepackage{amsmath,amsfonts,amssymb,latexsym,amsthm,dsfont}
\usepackage[a4paper,margin=2.5cm]{geometry}
\usepackage{graphicx}
\usepackage[backref]{hyperref}
\usepackage[utf8]{inputenc}
\usepackage[T1]{fontenc}
\usepackage{lmodern}
\usepackage{tikz,enumitem}
\usetikzlibrary{arrows,calc,decorations.text}
\usepackage[english]{babel}

\newtheorem{thm}{Theorem}[section]%
\newtheorem{cor}[thm]{Corollary}%
\newtheorem{lem}[thm]{Lemma}%

\newtheorem{xpl}[thm]{Example}%
\newtheorem{hyp}[thm]{Hypothesis}%
\newtheorem{rem}[thm]{Remark}%
\newtheorem{conj}[thm]{Conjecture}%
\newtheorem{ques}{Open question}

\newcommand{\dE}{\mathbb{E}}

\newcommand{\dN}{\mathbb{N}}
\newcommand{\dP}{\mathbb{P}}
\newcommand{\dR}{\mathbb{R}}

\newcommand{\cB}{\mathcal{B}}

\newcommand{\cL}{\mathcal{L}}
\newcommand{\cM}{\mathcal{M}}
\newcommand{\cP}{\mathcal{P}}

\newcommand{\cW}{\mathcal{W}}

\newcommand{\ABS}[1]{{{\left| #1 \right|}}} 

\newcommand{\BRA}[1]{{{\left\{#1\right\}}}} 
\newcommand{\DP}[1]{{{\left<#1\right>}}} 
\newcommand{\NRM}[1]{{{\left\| #1\right\|}}} 
\newcommand{\PAR}[1]{{{\left(#1\right)}}} 
\newcommand{\SBRA}[1]{{{\left[#1\right]}}} 
\renewcommand{\leq}{\leqslant}
\renewcommand{\geq}{\geqslant}

\newcommand{\ind}{\mathds{1}}

\title{Some simple but challenging Markov processes}
\author{Florent Malrieu} %

\date{\today}

\begin{document}

\maketitle

\begin{abstract}
In this note, we present few examples of Piecewise Deterministic 
Markov Processes and their long time behavior. They share two
important features: they are related to concrete models (in biology, 
networks, chemistry,\ldots) and they are mathematically rich. Their 
mathematical study relies on coupling method, spectral decomposition, 
PDE technics, functional inequalities. We also relate these simple examples  
to recent and open problems.  
\end{abstract}

\section{Introduction}

A Piecewise deterministic Markov processes (PDMP\footnote{This may also mean 
"Persi Diaconis: Mathemagician and Popularizer".}) is  
a stochastic process involving deterministic motion punctuated by 
random jumps. This large class of non diffusion stochastic models  
was introduced in the literature by Davis \cite{davis,MR1283589} 
(see also \cite{jacobsen}). As it will be stressed below, these processes 
arise naturally in many application areas: biology, communication 
networks, reliability of complex systems for example. From a mathematical 
point of view, they are simple to define but their study may require a broad 
spectrum of tools as stochastic coupling, functional inequalities, spectral 
analysis, dynamical systems, partial differential equations. 

The aim of the present paper is to present simple examples of PDMP appearing 
in different applied frameworks and to investigate their long time behavior. Rather 
than using generic technics (as Meyn-Tweedie-Foster-Lyapunov\ldots strategy) we 
will focus on as explicit as possible estimates. Several open and motivating questions 
(stability criteria, regularity of the invariant measure(s), explicit rate of convergence\ldots) 
are also listed along the paper. 
 
Roughly speaking the dynamics of a PDMP on a set $E$ 
depends on three local characteristics, namely, a flow
$\varphi$, a jump rate $\lambda$ and a transition kernel $Q$. Starting from
$x$, the motion of the process follows the flow $t\mapsto\varphi_t(x)$ until the first
jump time $T_1$ which occurs in a Poisson-like fashion
with rate $\lambda(x)$. More precisely, the distribution of the first jump time is 
given by 
\[
\dP_x(T_1>t)= \exp\PAR{-\int_0^t \! \lambda(\varphi_s(x))\,ds}.
\]
Then, the location of the process at the
jump time $T_1$ is selected by the transition measure $Q(\varphi_{T_1}(x),\cdot)$ and
the motion restarts from this new point as before. This motion is summed 
up by the infinitesimal generator: 
 \begin{equation}\label{eq:first-def-gi}
Lf(x)=F(x) \cdot \nabla f(x)+\lambda(x)\int_E\!(f(y)-f(x))\,Q(x,dy),
\end{equation}
where $F$ is the vector field associated to the flow $\varphi$. In several examples, the process 
may jump when it hits the boundary of $E$. The boundary of the space
$\partial E$ can be seen as a region where the jump rate is infinite (see for example \cite{comets} 
for the study of billiards in a general domain with random reflections). 

In the sequel, we denote by $\cP(\dR^d)$ the set of probability measures 
on $(\dR^d,\cB(\dR^d))$ and, for any $p\geq 1$, by $\cP_p(\dR^d)$ the 
set of probability measures on $(\dR^d,\cB(\dR^d))$ with a finite $p^{th}$-moment: 
$\mu\in\cP_p(\dR^d)$ if 
\[
\int_{\dR^d}\!\!\ABS{x}^p \,\mu(dx)<+\infty.
\] 
The total variation distance on $\cP(\dR^d)$ is given by 
\begin{align*}
\NRM{\nu-\tilde \nu}_{\mathrm{TV}}
&=\inf\BRA{\dP(X\neq  \tilde X)\, : \, X\sim \nu,\, \tilde X\sim\tilde \nu}\\
&=\sup\BRA{\int\! f\,d\nu-\int\! f\,d\tilde \nu \, :\, f\text{ bounded by }1/2}. 
\end{align*}
If $\nu$ and $\tilde \nu$ are absolutely continuous with respect to $\mu$ with 
density functions $g$ and $\tilde g$, then 
\[
\NRM{\nu-\tilde \nu}_{\mathrm{TV}}=\frac{1}{2}\int_{\dR^d}\!\!\ABS{g-\tilde g}\,d\mu.
\]
For $p\geq 1$, the Wasserstein distance of order $p$, defined on $\cP_p(\dR^d)$, is given by 
\[
W_p(\nu,\tilde \nu)
=\inf\BRA{\SBRA{\dE\PAR{\ABS{X-\tilde X}^p}}^{1/p}\, : \, X\sim \nu,\, \tilde X\sim\tilde \nu}.
\]
Similarly to the total variation distance, the Wasserstein distance of order $1$ has 
a nice dual formulation: 
 \[
W_1(\nu,\tilde \nu)
=\sup\BRA{\int\! f\,d\nu-\int\! f\,d\tilde \nu \, :\, f\text{ is $1$-Lipschitz}}.
\]
A generic dual expression can be formulated for $W_p$ (see \cite{MR1964483}).

\section{Storage models, with a bandit\ldots}

Let us consider the PDMP driven by the following infinitesimal generator: 
\[
Lf(x)=-\beta xf'(x)+\alpha \int_0^\infty\! (f(x+y)-f(x))\,e^{-y}dy.
\]  
Such processes appear in the modeling of storage problems or pharmacokinetics that describe 
the evolution of the concentration of a chemical product in the human body.
The present example is studied in \cite{RT,BCGMZ}. More realistic models are studied in 
\cite{BCT,bouguet}. Similar processes can also be used as stochastic gene expression 
models (see \cite{MR3097042,MR3175198}).

In words, the current stock $X_t$ decreases exponentially at rate $\beta$, and 
increases at random exponential times by a random (exponentially distributed) amount.
Let us introduce a Poisson process ${(N_t)}_{t\geq 0}$ with 
intensity~$\alpha$ and jump times ${(T_i)}_{i\geq 0}$ (with $T_0=0$) and a 
sequence~${(E_i)}_{i\geq 1}$ of independent random variables 
with an exponential law of parameter $1$ independent of ${(N_t)}_{t\geq 0}$. 
The process ${(X_t)}_{t\geq 0}$ starting from $x\geq 0$ can be 
constructed as follows: for any $i\geq 0$, 
\[
X_t=
\begin{cases}
e^{-\beta (t-T_i)}X_{T_i}&\text{if }T_i\leq t<T_{i+1},\\
e^{-\beta (T_{i+1}-T_i)}X_{T_i}+E_{i+1} &\text{if }t=T_{i+1}. 
\end{cases}
\]
This model is sufficiently naïve to express the Laplace transform of $X$. 

\begin{lem}[Laplace transform]
For any $t\geq 0$ and $s<1$, the Laplace transform of $X_t$ is given by  
\[
L(t,s):=\dE\PAR{e^{s X_t}}=L(0,s e^{-\beta t})\PAR{\frac{1-s e^{-\beta t}}{1-s}}^{\alpha/\beta}, 
\] 
where $L(0,\cdot)$ stands for the Laplace transform of $X_0$. 
In particular, the invariant distribution of $X$ is the Gamma distribution with density 
\[
x\mapsto \frac{x^{{\alpha/\beta}-1}e^{-x}}{\Gamma(\alpha/\beta)}\ind_{[0,+\infty)}(x). 
\]
\end{lem}

\begin{proof}
 Applying the infinitesimal generator to $x\mapsto e^{s x}$, one deduces that 
the function $L$ is solution of the following partial differential equation: 
\[
\partial_tL(t,s)=-\beta s\partial_sL(t,s)+\frac{\alpha s}{1-s}L(t,s). 
\] 
More generally, if the random income is non longer exponentially distributed but has 
a Laplace transform $L_i$ then $L$ is solution of 
\[
\partial_tL(t,s)=-\beta s\partial_sL(t,s)+ \alpha (L_i(s)-1)L(t,s). 
\] 
As a consequence, if $G$ is given by $G(t,s)=\log L(t,s)+(\alpha/\beta)\log(1-s)$ then 
\[
\partial_t G(t,s)=-\beta s\partial_s G(t,s). 
\]
The solution of this partial differential equation is given by $G(t,s)=G(0,se^{-\beta t})$.
\end{proof}
The next step is to investigate the convergence to equilibrium. 
\begin{thm}[Convergence to equilibrium]\label{prop:storage}
Let us denote by $\nu P_t$ the law of $X_t$ if $X_0$ is distributed according to $\nu$. 
For any $x,y\geq 0$ and $t\geq 0$ and $p\geq 1$,
\[
W_p(\delta_xP_t,\delta_yP_t)\leq \ABS{x-y}e^{-\beta t},
\]
and (when $\alpha\neq\beta$)
\begin{equation}\label{eq:storage-TV}
\NRM{\delta_xP_t-\delta_yP_t}_{\mathrm{TV}}\leq 
e^{-\alpha t}+\ABS{x-y} \alpha\frac{e^{-\beta t}-e^{-\alpha t}}{\alpha-\beta}.
\end{equation}
Moreover, if $\mu$ is the invariant measure of the process $X$, 
we have for any probability measure $\nu$ with a finite first moment and 
$t\geq 0$,  
\[
\NRM{\nu P_t-\mu }_{\mathrm{TV}}\leq 
\NRM{\nu -\mu }_{\mathrm{TV}} e^{-\alpha t}+
W_1(\nu,\mu)\alpha\frac{e^{-\beta t}-e^{-\alpha t}}{\alpha-\beta}. 
\]
\end{thm}

\begin{rem}[Limit case]
  In the case $\alpha = \beta$, the upper bound \eqref{eq:storage-TV} 
  becomes
\begin{equation*}
\NRM{\delta_xP_t-\delta_yP_t}_{\mathrm{TV}}\leq 
(1 + \ABS{x-y} \alpha t )e^{-\alpha t}.
\end{equation*}
\end{rem}

\begin{rem}[Optimality]
 Applying $L$ to the test function $f(x)=x^n$ allows us to compute recursively 
 the moments of $X_t$. In particular, 
 \[
\dE_x(X_t)=\frac{\alpha}{\beta}+\PAR{x-\frac{\alpha}{\beta}}e^{-\beta t}.
\]
This relation ensures that the rate of convergence for the Wasserstein distance 
is sharp. Moreover, the coupling for the total variation distance requires at least one 
jump. As a consequence, the exponential rate of convergence is greater than 
$\alpha$. Thus, Equation~\eqref{eq:storage-TV} provides the optimal rate 
of convergence $\alpha\wedge\beta$. 
\end{rem}

\begin{proof}[Proof of Theorem~\ref{prop:storage}]
Firstly, consider two processes $X$ and $Y$
starting respectively at $x$ and $y$ and driven by the same 
randomness (\emph{i.e.} Poisson process and jumps). Then the 
distance between $X_t$ and $Y_t$ is deterministic: 
\[
X_t-Y_t=(x-y) e^{-\beta t}. 
\] 
Obviously, for any $p\geq 1$ and $t\geq 0$, 
\[
W_p(\delta_xP_t,\delta_yP_t)\leq |x-y| e^{-\beta t}.
\]
Let us now construct explicitly a coupling at time $t$ to get the upper 
bound \eqref{eq:storage-TV} for the total variation distance. The jump times 
of ${(X_t)}_{t\geq 0}$ and ${(Y_t)}_{t\geq 0}$ are the ones of a 
Poisson process ${(N_t)}_{t\geq 0}$ with intensity $\alpha$ and 
jump times ${(T_i)}_{i\geq 0}$. Let us now construct the jump heights 
${(E^X_i)}_{1\leq i\leq N_t}$ and ${(E^Y_i)}_{1\leq i\leq N_t}$ of 
$X$ and $Y$ until time $t$. If $N_t=0$, no jump occurs. If $N_t\geq 1$, 
we choose $E^X_i=E^Y_i$ for $1\leq i\leq N_t-1$ and
$E^X_{N_t}$ and $E^Y_{N_t}$ in order to maximise the probability 
\[
\dP\PAR{X_{T_{N_t}}+E^X_{N_t}=Y_{T_{N_t}}+E^Y_{N_t}
\big\vert X_{T_{N_t}},Y_{T_{N_t}}}.
\]
This maximal probability of coupling is equal to 
\[
\exp\PAR{-|X_{T_{N_t}}-Y_{T_{N_t}}|}
=\exp\PAR{-|x-y|e^{-\beta T_{N_t}}}
\geq 1-|x-y|e^{-\beta T_{N_t}}.
\]
As a consequence, we get that 
\begin{align*}
\NRM{\delta_xP_t-\delta_yP_t}_{\mathrm{TV}}&\leq 
1-\dE\SBRA{\PAR{1-|x-y|e^{-\beta T_{N_t}}}\ind_\BRA{N_t\geq 1}}\\
&\leq e^{-\alpha t}+|x-y|\dE\PAR{e^{-\beta T_{N_t}}\ind_\BRA{N_t\geq 1}}. 
\end{align*}
The law of $T_n$ conditionally on the event $\BRA{N_t=n}$ has the density 
\[
u\mapsto n \frac{u^{n-1}}{t^n}\ind_{[0,t]}(u).
\] 
This ensures that
\[
\dE\PAR{e^{-\beta T_{N_t}}\ind_\BRA{N_t\geq 1}}
=\int_0^1\! e^{-\beta t v}\dE\PAR{N_tv^{N_t-1}}\,dv.
\]
Since the law of $N_t$ is the Poisson distribution with parameter $\alpha t$, one 
has 
\[
\dE\PAR{N_tv^{N_t-1}}=\alpha t e^{\alpha t(v-1)}. 
\]
This ensures that 
\[
\dE\PAR{e^{-\beta N_t}\ind_\BRA{N_t\geq 1}}
=\alpha \frac{e^{-\beta t}-e^{-\alpha t}}{\alpha-\beta},
\] 
which completes the proof. Finally, to get the last estimate, we proceed 
as follows: if $N_t$ is equal to 0, a coupling in total variation of the initial 
measures is done, otherwise, we use the coupling above.  
\end{proof}

\begin{rem}[Another example]
Surprisingly, a process of the same type appears in \cite{lamberton-pages} in the study of 
the so-called bandit algorithm. The authors have to investigate the long time behavior of the 
process driven by  
\[
Lf(y)=(1-p-py)f'(y)+q y\frac{f(y+g)-f(y)}{g},
\]
where $0<q<p<1$ and $g>0$. This can be done following the lines of the proof of 
Theorem~\ref{prop:storage}.
\end{rem}

\section{The TCP model with constant jump rate}

This section is devoted to the process on $[0,+\infty)$ driven by the following 
infinitesimal generator 
\[
Lf(x)=f'(x)+\lambda(f(x/2)-f(x))
\quad(x\geq 0). 
\]
In other words, the process grows linearly between jump times that are the 
one of a homogeneous Poisson process with parameter $\lambda$ and it is divided by 
2 at these instants of time. See Section~\ref{sec:geneTCP} for concrete motivations 
and generalizations.

\subsection{Spectral decomposition} 

Without loss of generality, we choose $\lambda=1$ in this section. 
The generator $L$ of the naïve TCP process preserves the degree of polynomials. 
As a consequence, for any $n\in\dN$, the eigenvalue $\lambda_n=-(1-2^{-n})$ is 
associated to a polynomials $P_n$ with degree $n$. As an example, 
\[
P_0(x)=1,\quad P_1(x)=x-2 \quad\text{and}\quad P_2(x)=x^2-8x+32/3.
\]
Moreover, one can explicitly compute 
the moments of the invariant measure $\mu$ (see \cite{LL}): for any $n\in\dN$
\[
\int \! x^n \,\mu(dx)=\frac{n!}{\prod_{k=1}^n(1-2^{-k})}.
\]
Roughly speaking, this relation comes from the fact that the functions 
$m_n: t\in[0,\infty)\mapsto\dE(X_t^n)$ for $n\geq 0$ are solution of 
\[
m_n'(t)=nm_{n-1}(t)+\PAR{2^{-n}-1}m_n(t).
\]
It is also shown in \cite{DGR} that the Laplace transform of $\mu$ is finite on 
a neighborhood of the origin. As a consequence, the polynomials are dense 
in $L^2(\mu)$. Unfortunately, the eigenvectors of $L$ are not orthogonal in 
$L^2(\mu)$. For example, 
\[
\int\! P_1P_2\,d\mu=-\frac{64}{27}. 
\]
This lack of symmetry (due to the fact that the invariant measure $\mu$ is not reversible) 
prevents us to easily deduce an exponential convergence to equilibrium in $L^2_\mu$.

When the invariant measure is reversible, the spectral decomposition (and particularly its 
spectral gap) of $L$ provides fine estimates for the convergence to equilibrium. See for 
example \cite{LPW} and the connection with coupling strategies and strong stationary times 
introduced in \cite{aldous-diaconis}.

\begin{ques}[Spectral proof of ergodicity]
Despite the lack of reversibility, is it possible to use the spectral properties of $L$ to get 
some estimates on the long time behavior of $X$?
\end{ques}

\begin{rem}
This spectral approach has been fruitfully used in \cite{gadat-miclo,Miclo-Monmarche} to 
study (nonreversible) hypocoercive models. 
\end{rem}

\subsection{Convergence in Wasserstein distances}

The convergence in Wasserstein distance is obvious. 
\begin{lem}[Convergence in Wasserstein distance \cite{PR,CMP10}]
For any $p\geq 1$, 
\begin{equation}\label{eq:Wasserstein-coupling}
W_p(\delta_xP_t,\delta_yP_t)\leq \ABS{x-y} e^{-\lambda_pt}
\quad\text{with}\quad
\lambda_p=\frac{\lambda(1-2^{-p})}{p}. 
 \end{equation}
\end{lem}
\begin{rem}[Alternative approach]
 The case $p=1$ is obtained in \cite{PR} by PDEs estimates using the 
 following alternative formulation of the Wasserstein distance on $\dR$. 
 If the cumulative distribution functions of the two probability 
 measures $\nu$ and $\tilde \nu$ are $F$ and $\tilde F$ then 
 \[
W_1(\nu,\tilde \nu)=\int_{\dR}\! \vert F(x)-\tilde F(x)\vert \,dx. 
\] 
\end{rem}
The general case $p\geq 1$ is obvious from the probabilistic point of view: 
choosing the same Poisson process ${(N_t)}_{t\geq 0}$ to drive 
the two processes provides that the two coordinates jump simultaneously 
and 
\[
\ABS{X_t-Y_t}=\ABS{x-y}{2}^{-N_t}. 
\]
As a consequence, since the law of $N_t$ is the Poisson distribution with 
parameter $\lambda t$, one has 
\[
\dE_{x,y}\PAR{\ABS{X_t-Y_t}^p}
=\ABS{x-y}^p \dE\PAR{2^{-pN_t}}
=\ABS{x-y}^p e^{-p\lambda_pt}.
\]
This coupling turns out to be sharp. Indeed, one can compute explicitly 
the moments of~$X_t$ (see \cite{LL,MR2426601}): for every 
$n\geq0$, every $x\geq0$, and every $t\geq0$,
\begin{equation}\label{eq:momlop}
  \dE_x(X_t^n)=\frac{n!}{\prod_{k=1}^n\theta_k}+%
  n!\sum_{m=1}^n\biggr(\sum_{k=0}^m\frac{x^k}{k!}%
  \prod_{\substack{j=k\\j\neq m}}^n\frac{1}{\theta_j-\theta_m}\biggr)e^{-\theta_mt},
\end{equation}
where $\theta_n=\lambda(1-2^{-n})=n\lambda_n$ for any 
$n\geq1$. Obviously, assuming for example that $x>y$,
\begin{align*}
W_n(\delta_xP_t,\delta_yP_t)^n
&\geq \dE_x((X_t)^n)-\dE_y((Y_t)^n)\\
&\underset{t\to \infty}{\sim}
 n!\biggr(\sum_{k=0}^n\frac{x^k-y^k}{k!}%
  \prod_{\substack{j=k}}^{n-1}\frac{1}{\theta_j-\theta_n}\biggr)e^{-\theta_nt}.
\end{align*}
As a consequence, the rate of convergence in 
Equation~\eqref{eq:Wasserstein-coupling} is optimal for any $n\geq 1$.

\subsection{Convergence in total variation distance}

The estimate for the Wasserstein rate of convergence 
does not provide on its own any information about the total variation 
distance between $\delta_xP_t$ and $\delta_yP_t$. It turns out that 
this rate of convergence is the one of the $W_1$ distance. This is 
established in~\cite[Thm~1.1]{PR}. Let us provide here an improvement 
of this result by a probabilistic argument.

\begin{thm}[Convergence in total variation distance]
\label{pr:constantTCP}
For any $x,y\geq 0$ and $t\geq 0$, 
\begin{equation}\label{eq:constant-TV-xy}
\NRM{\delta_xP_t-\delta_yP_t}_{\mathrm{TV}}\leq 
\lambda e^{-\lambda t/2}\ABS{x-y}+e^{-\lambda t}. 
\end{equation}
As a consequence, for any measure $\nu$ with a finite first moment 
and $t\geq 0$, 
\begin{equation}\label{eq:constant-TV-ergo}
 \NRM{\nu P_t-\mu}_{\mathrm{TV}}\leq 
\lambda e^{-\lambda t/2}W_1(\nu,\mu)+
 e^{-\lambda t} \NRM{\nu-\mu}_{\mathrm{TV}}.
\end{equation}
\end{thm}

\begin{rem}[Propagation of the atom]
  Note that the upper bound obtained in Equation \eqref{eq:constant-TV-xy} does not go to zero as 
  $y\to x$. This is due to the fact that $\delta_xP_t$ has an atom at $y+t$ with mass 
  $e^{-\lambda t}$. 
\end{rem}

\begin{proof}[Proof of Theorem \ref{pr:constantTCP}]
The coupling is a slight modification of the Wasserstein one.  
The paths of ${(X_s)}_{0\leq s\leq t}$ and ${(Y_s)}_{0\leq s\leq t}$ 
starting respectively from $x$ and $y$ are determined by their jump 
times ${(T^X_n)}_{n\geq 0}$ and ${(T^Y_n)}_{n\geq 0}$ up to time $t$. These
sequences have the same distribution than the jump times of a Poisson 
process with intensity $\lambda$. 

Let ${(N_t)}_{t\geq 0}$ be a Poisson process with intensity $\lambda$ 
and ${(T_n)}_{n\geq 0}$ its jump times with the convention $T_0=0$.
Let us now construct the jump times of $X$ and $Y$. Both  
processes make exactly $N_t$ jumps before time $t$. If $N_t=0$, then 
\[
X_s=x+s\quad \text{and}\quad Y_s=y+s\quad \text{for }0\leq s\leq t.
\]
Assume now that $N_t\geq 1$. The $N_t-1$ first jump times of $X$ and $Y$ 
are the ones of ${(N_t)}_{t\geq 0}$: 
\[
T^X_k=T^Y_k=T_k \quad 0\leq k\leq N_t-1.  
\] 
In other words, the Wasserstein coupling 
acts until the penultimate jump time $T_{N_t-1}$. At that time, we have 
\[
X_{T_{N_t-1}}-Y_{T_{N_t-1}}=\frac{x-y}{2^{N_t-1}}.
\]
Then we have to define the last jump time for each process. If they are 
such that 
\[
T^X_{N_t}=T^Y_{N_t}+X_{T_{N_t-1}}-Y_{T_{N_t-1}},
\]
then the paths of $X$ and $Y$ are equal on the 
interval $(T^X_{N_t},t)$ and can be chosen to be equal for any time 
larger than $t$.  

Recall that conditionally on the event $\BRA{N_t=1}$, the law of 
$T_1$ is the uniform distribution on $(0,t)$. More generally, if  $n\geq 2$, 
conditionally on the set $\BRA{N_t=n}$, the law of the penultimate jump 
time $T_{n-1}$ has a density $s\mapsto n(n-1)t^{-n}(t-s)s^{n-2}\ind_{(0,t)}(s)$ 
and conditionally on the event $\BRA{N_t=n, T_{n-1}=s}$, the law of $T_n$ 
is uniform on the interval~$(s,t)$. 

Conditionally on $N_t=n\geq 1$ and $T_{n-1}$, $T^X_n$ and $T^Y_n$ 
are uniformly distributed on $(T_{n-1},t)$ and can be chosen such that  
\begin{align*}
&\dP\PAR{T^X_n=T^Y_n+\frac{x-y}{2^{n-1}}
\,\Big\vert\, N^X_t=N^Y_t=n,\,T^X_{n-1}=T^Y_{n-1}=T_{n-1}}\\
&\quad\quad=\PAR{1-\frac{|x-y|}{2^{n-1}(t-T_{n-1})}}\vee 0
\geq 1-\frac{|x-y|}{2^{n-1}(t-T_{n-1})}.
\end{align*}
This coupling provides that 
\begin{align*}
\NRM{\delta_xP_t-\delta_yP_t}_{\mathrm{TV}}&\leq 
1-\dE\SBRA{\PAR{1-\frac{|x-y|}{2^{N_t-1}(t-T_{N_t-1})}}\ind_\BRA{N_t\geq1}}\\
&\leq e^{-\lambda t}+|x-y|
\dE\PAR{\frac{2^{-N_t+1}}{(t-T_{N_t-1})}\ind_\BRA{N_t\geq 1}}.
\end{align*}
For any $n\geq 2$, 
\[
\dE\PAR{\frac{1}{t-T_{N_t-1}}\Big\vert N_t=n}
=\frac{n(n-1)}{t^n}\int_0^t\! u^{n-2}\,du=\frac{n}{t}.
\]
This equality also holds for $n=1$. Thus we get that 
\[
\dE\PAR{\frac{2^{-N_t+1}}{(t-T_{N_t-1})}\ind_\BRA{N_t\geq 1}}
=\frac{1}{t}\dE\PAR{N_t 2^{-N_t+1}}=\lambda e^{-\lambda t/2},
\]
since $N_t$ is distributed according to the Poisson law with 
parameter $\lambda t$. This provides the 
estimate \eqref{eq:constant-TV-xy}. The general 
case \eqref{eq:constant-TV-ergo} is a straightforward consequence: if $N_t$ 
is equal to 0, a coupling in total variation of the initial measures is done, otherwise, 
we use the coupling above.  
\end{proof}

\subsection{Some generalizations}\label{sec:geneTCP}

This process on $\dR_+$ belongs to the subclass of the AIMD (Additive Increase
Multiplicative Decrease) processes. Its infinitesimal generator is given by
\begin{equation}\label{eq:tcp}
Lf(x)=f'(x)+\lambda(x)\int_0^1\! (f(ux)-f(x))\,\nu(du), 
\end{equation}
where $\nu$ is a probability measure on $[0,1]$ and $\lambda$ is a non negative function. 
It can be viewed as the limit behavior of the congestion of a single channel (see
\cite{DGR,MR2023017} for a rigorous derivation of this limit). In
\cite{maulik-zwart}, the authors give a generalization of the scaling procedure to interpret
various PDMPs as the limit of discrete time Markov chains and in
\cite{MR2576022} more general increase and decrease profiles are 
considered as models for TCP. In the real world (Internet), the AIMD 
mechanism allows a good compromise between the minimization of 
network congestion time and the maximization of mean throughput. 
See also \cite{MR1928492} for a simplified TCP windows size model. See
\cite{MR2576022,MR2197972,MR2426601,OKM,MR2355580,MR2285697,
MR2164929} for other works dedicated to this process. Generalization 
to interacting multi-class transmissions are considered in 
\cite{MR2588247,GR2}. 

Such processes are also used to model the evolution of the size of bacteria or 
polymers which mixes growth and fragmentation: they growth in a deterministic 
way with a growth speed $x\mapsto \tau(x)$, 
and split at rate $x\mapsto \lambda(x)$ into two (for simplicity) parts $y$ and $x-y$ 
according a kernel $\beta(x,y)dy$. The infinitesimal generator associated to this dynamics 
writes 
\[
Lf(x)=\tau(x)f'(x)+\lambda(x)\int_0^x\!(f(y)-f(x))\beta(x,y)\,dy. 
\]
If the initial distribution of the size has a density $u(\cdot,0)$ then 
this density is solution of the following integro-differential PDE: 
\[
\partial_tu(x,t)=-\partial_x(\tau(x)u(x,t))-\lambda(x)u(x,t)+\int_x^\infty\! \lambda(y)\beta(y,x)u(y,t)\,dy.
\]
If one is interesting in the density of particles with size $x$ at time $t$ in the 
growing population (a splitting creates two particles), one has to consider the PDE 
\[
\partial_tu(x,t)=-\partial_x(\tau(x)u(x,t))-\lambda(x)u(x,t)
+2\int_x^\infty\! \lambda(y)\beta(y,x)u(y,t)\,dy.
\]
This growth-fragmentation equations have been extensively studied from a PDE 
point of view  (see for example \cite{perthame,DG,MR2935368,mischler}). A probabilistic 
approach is used in \cite{MR2932439} to study the pure fragmentation process.


\section{Switched flows and motivating examples}

Let $E$ be the set $\BRA{1,2,\ldots,n}$, $({\lambda(\cdot,i,j)})_{i,j\in E}$ be
nonnegative continuous functions on $\dR^d$, and, for any $i\in E$, 
$F^i(\cdot): \dR^d\mapsto \dR^d$ be a smooth vector field such that 
the ordinary differential equation
\[
\begin{cases}
 \dot x_t=F^i(x_t) &\text{ for }t>0,\\
 x_0=x
\end{cases}
\] 
has a unique and global solution $t\mapsto \varphi^i_t(x)$ on $[0,+\infty)$ for
any initial condition $x\in\dR^d$. Let us consider the Markov process
\[
{(Z_t)}_{t\geq 0}={((X_t,I_t))}_{t\geq 0} 
\text{ on }\dR^d\times E 
\]
defined by its infinitesimal generator $L$ as follows: 
\[
  Lf(x,i)=
  F^i(x)\cdot \nabla_x f(x,i)+
  \sum_{j\in E} \lambda(x,i,j)(f(x,j)-f(x,i)) 
\]
for any smooth function $f:\ \dR^d\times E\rightarrow \dR$. 

These PDMP are also known as hybrid systems. They have been intensively studied 
during the past decades (see for example the review~\cite{YZ}). In particular, they 
naturally appear as the approximation of Markov chains mixing slow and fast dynamics
(see \cite{CDMR}). They could also be seen as a continuous time
version of iterated random functions (see the excellent review~\cite{diaconis}).

In this section, we 
present few examples from several applied areas and describe their long time behavior. 

\subsection{A surprising blow up for switched ODEs}\label{sec:blowup}

The main probabilistic results of this section are established in 
\cite{mattingly}. Consider the Markov process $(X,I)$ on $\dR^2\times\BRA{0,1}$ 
driven by the following infinitesimal generator: 
\begin{equation}\label{eq:generator}
Lf(x,i)=(A_ix)\cdot \nabla_x f(x,i)+r (f(x,1-i)-f(x,i)) 
\end{equation}
where $r>0$ and $A_0$ and $A_1$ are the two following matrices
\begin{equation}\label{eq:defA}
A_0=\begin{pmatrix}
 -\alpha & 1\\
 0 & -\alpha
\end{pmatrix}
\quad\text{and}\quad
A_1=\begin{pmatrix}
 -\alpha & 0\\
 -1 & -\alpha
\end{pmatrix}
\end{equation}
for some positive $\alpha$. In other words, ${(I_t)}_{t\geq 0}$ is a Markov process on $\BRA{0,1}$ with 
constant jump rate $r$ (from 0 to 1 and from 1 to 0) and ${(X_t)}_{t\geq 0}$ is the solution of 
$\dot X_t= A_{I_t}X_t$. 

The two matrices $A_0$ and $A_1$ are Hurwitz matrices (all eigenvalues have strictly 
negative real parts). Moreover, it is also the case for the matrix $A_p=pA_1+(1-p)A_0$ 
with $p\in[0,1]$ since the eigenvalues of $A_p$ are $-\alpha\pm i\sqrt{p(1-p)}$. 
Then, for any $p\in[0,1]$, there exists $K_p\geq 1$ and $\rho>0$ such that 
\[
\NRM{x_t}\leq K_p \NRM{x_0} e^{-\rho t}, 
\]
for any solution ${(x_t)}_{t\geq 0}$ of $\dot x_t=A_p x_t$. 



\subsubsection{Asymptotic behavior of the continuous component}

The first step is to use polar coordinates to study the large time behavior of 
$R_t=\NRM{X_t}$ and $U_t$ the point on the unit circle $S^1$ given by $X_t/R_t$. 
One gets that 
\begin{align*}
 \dot{R}_t&=R_t\DP{A_{I_t}U_t,U_t}\\
 \dot{U}_t&=A_{I_t}U_t-\DP{A_{I_t}U_t,U_t}U_t. 
\end{align*}
As a consequence, $(U_t,I_t)$ is a Markov process on $S^1\times \BRA{0,1}$. 
One can show that it admits a unique invariant measure $\mu$.Therefore, if 
$\dP(R_0=0)=0$, 
\[
\frac{1}{t}\log R_t=\frac{1}{t}\log R_0+ \frac{1}{t}\int_0^t \! \DP{A_{I_s}U_s,U_s}\,ds
\xrightarrow[t\to\infty]{a.s.}\int \DP{A_iu,u}\mu(du,i).
\]
The stability of the Markov process depends on the sign of 
\[
L(\alpha,r):=\int \DP{A_iu,u}\mu(du,i).
\]
An "explicit" formula for $L(\alpha,r)$ can be formulated in terms of the classical 
trigonometric functions 
\[
\cot(x)=\frac{\cos(x)}{\sin(x)}, 
\quad 
\sec(x) =\frac{1}{\cos(x)}
\quad\text{and}\quad
\csc(x)=\frac{1}{\sin(x)}.
\]

\begin{thm}[Lyapunov exponent \cite{mattingly}]\label{prop:mat2}
For any $r>0$ and $\alpha>0$, 
\[
L(\alpha,r)=G(r)-\alpha
\quad\text{where}\quad
G(r)=\int_0^{2\pi} (p_0(\theta;r)-p_1(\theta;r))\cos(\theta)\sin(\theta)\,d\theta>0
\]
and $p_0$ and $p_1$ are defined as follows: for $\theta\in (-\pi/2,0)$

\begin{align*}
H(\theta;r)&=\exp(-2r \cot(2\theta))\int_\theta^0\!\exp(2r \cot(2y))\sec^2(y)\,dy,\\
C(r)&=\SBRA{4\int_{-\frac{\pi}{2}}^0\! \sec^2(x)+(\csc^2(x)-\sec^2(x))r H(x;r)\,dx}^{-1},\\
p_0(\theta;r)&=C(r)\csc^2(\theta) r H(\theta;r),\\
p_1(\theta;r)&=C(r)\sec^2(\theta) [1-r H(\theta;r)],
\end{align*}
and for any $\theta\in\dR$, 
\[
p_i(\theta;r)=p_{1-i}(\theta+\pi/2;r)=p_i(\theta+\pi;r).
\]
\end{thm}

\begin{proof}[Sketch of proof of Theorem~\ref{prop:mat2}]
Let us denote by ${(\Theta_t)}_{t\geq 0}$ the lift of ${(U_t)}_{t\geq 0}$. 
The process $(\Theta,I)$ is also Markovian. Moreover, its infinitesimal 
generator is given by 
\[
\cL f(\theta,i)=
-\SBRA{i\cos^2(\theta)+(1-i)\sin^2(\theta)}\partial_\theta f(\theta,i)+r\SBRA{f(\theta,1-i)-f(\theta,i)}.
\]
Notice that the dynamics of $(\Theta,I)$ does not depend on the parameter $\alpha$. 
This process has a unique invariant measure $\mu$ (depending on the jump rate~$r$). 
With the one-to-one correspondence between a point on $S^1$ and 
a point in $[0,2\pi)$, let us write the invariant probability 
measure $\mu$ as
\[
\mu(d\theta,i)=p_i(\theta;r)\ind_{[0,2\pi)}(\theta)\,d\theta, 
\]
 The functions $p_0$ and $p_1$ are solution of 
\[
\begin{cases}
\partial_\theta (\sin^2(\theta)p_0(\theta))+r(p_1(\theta)-p_0(\theta))=0,\\
\partial_\theta(\cos^2(\theta)p_1(\theta))+r(p_0(\theta)-p_1(\theta))=0. 
\end{cases}
\] 
These relations provide the desired expressions.
\end{proof}
The previous technical result provides immediately the following result on the 
(in)stability of the process. 
\begin{cor}[(In)Stability \cite{mattingly}]\label{prop:mattingly}
There exist $\alpha>0$, $a>0$ and $b>0$ such that 
$L(\alpha,r)$ is negative if $r<a$ or $r>b$ and $L(\alpha,r)$ is positive for some $r\in (a,b)$. 
\end{cor}

From numerical experiments, see Figure~\ref{fi:G}, one can formulate the 
following conjecture on the function $G$. 
\begin{figure}
\begin{center}
 \includegraphics[scale=.5]{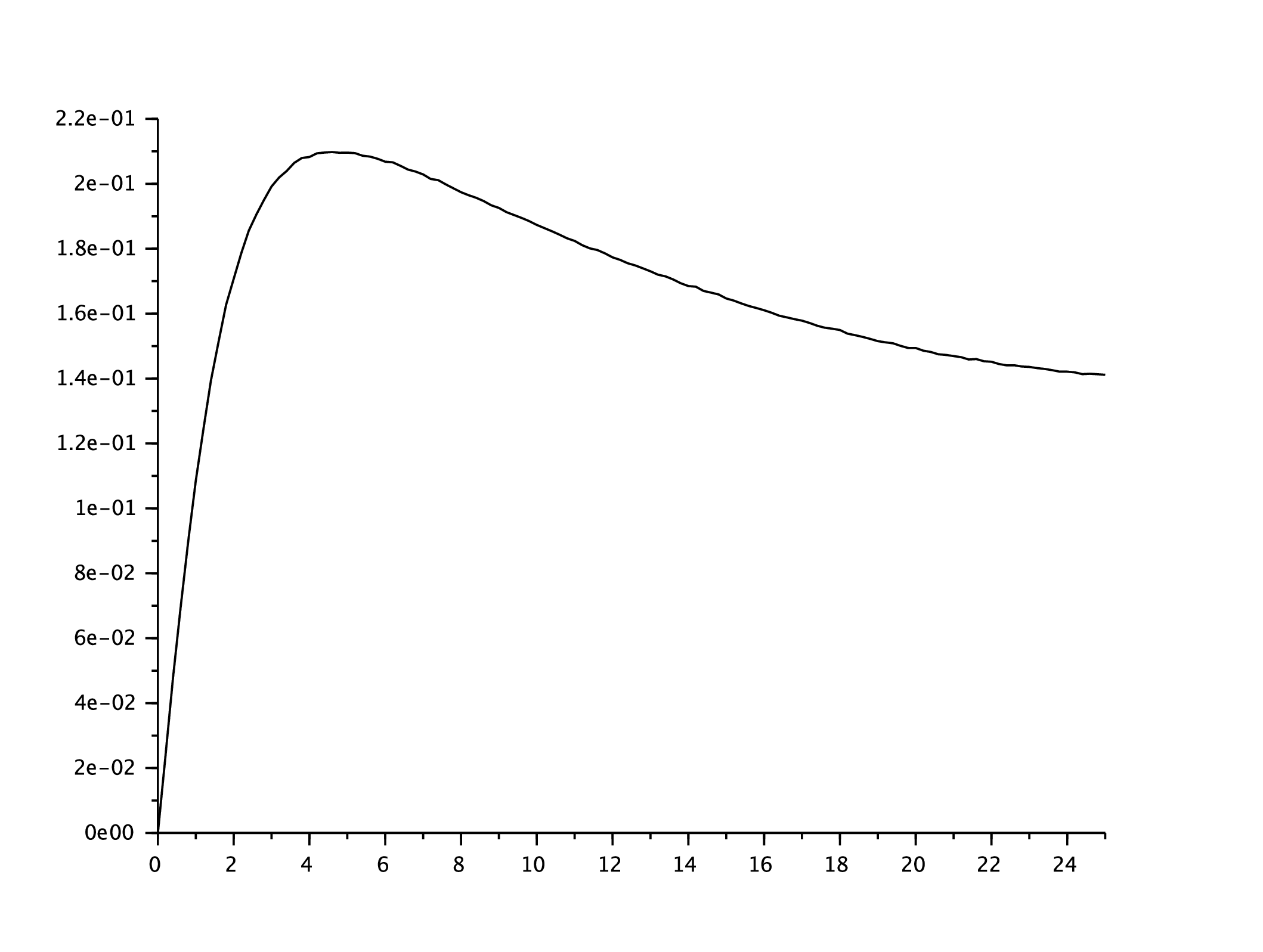}
 \caption{Shape of the function $G$ defined in Theorem~\ref{prop:mat2}.}
 \label{fi:G}
\end{center}
\end{figure}

\begin{conj}[Shape of $G$]\label{conj:G}
There exists $r_c\sim 4.6$ such that $G'(r)>0$ for $r<r_c$ and $G'(r)<0$ for $r>r_c$ and 
$G(r_c)\sim 0.2$. Moreover, 
\[
\lim_{r\to 0}G(r)=0
\quad\text{and}\quad
\lim_{r\to \infty}G(r)=0.
\]
\end{conj}

\begin{ques}[Shape of the instability domain]
Is it possible to prove Conjecture~\ref{conj:G}?
This would imply that the set 
\[
U_\alpha=\BRA{r>0\,:\, \NRM{X_t}\xrightarrow[t\to\infty]{p.s.}+\infty}
=\BRA{r>0\,:\, L(r,\alpha)>0}=\BRA{r>0\,:\, G(r)>\alpha} 
\]
is empty for $\alpha>G(r_c)$ and is a non empty interval if $\alpha<G(r_c)$. 
\end{ques}

\begin{rem}[On the irreducibility of $(U,I)$]
Notice that one can modify the matrices $A_0$ and $A_1$ in such a way 
that $(U,I)$ has two ergodic invariant measures (see \cite{BLMZ1}). 
\end{rem}

\begin{ques}[Oscillations of the Lyapunov exponent]
Is it possible to choose the two $2\times 2$ matrices $A_0$ 
and $A_1$ in such a way that the set of jump rates $r$ associated to unstable processes 
is the union of several intervals? 
\end{ques}



\subsubsection{A deterministic counterpart}

Consider the following ODE 
\begin{equation}\label{eq:deterministic}
\dot{x}_t=(1-u_t)A_0x_t+u_tA_1x_t, 
\end{equation}
where $u$ is a given measurable function from $[0,\infty)$ to $\BRA{0,1}$. 
The system is said to be \emph{unstable} if there exists a starting point 
$x_0$ and a measurable function $u$ : $[0,\infty)\to \BRA{0,1}$ such that the solution 
of \eqref{eq:deterministic} goes to infinity. 

In \cite{boscain,balde-boscain,1}, the authors provide necessary and 
sufficient conditions for the solution of \eqref{eq:deterministic} to be 
unbounded for two matrices $A_0$ and $A_1$ in $\cM_2(\dR)$. In the particular 
case \eqref{eq:defA}, this result reads as follows. 

\begin{thm}[Criterion for stability \cite{1}]\label{th:BBM}
If $A_0$ and $A_1$ are given by~\eqref{eq:defA}, the system \eqref{eq:deterministic} 
is unbounded if and only if 
\begin{equation}\label{eq:R}
R(\alpha^2):=\frac{1+2\alpha^2+\sqrt{1+4\alpha^2}}{2\alpha^2} e^{-2\sqrt{1+4\alpha^2}}>1.
\end{equation}
\end{thm}
More precisely, the result in \cite{1} ensures that 
\begin{itemize}
 \item if $2\alpha>1$ (case $S1$ in \cite{1}) then there exists a common 
 quadratic Lyapunov function for $A_0$ and $A_1$ (and $\NRM{X_t}$ 
 goes to 0 exponentially fast as $t\to\infty$ for any function $u$),
 \item if $2\alpha\leq 1$ (case $S4$ in \cite{1}) then, the system is 
  \begin{itemize}
\item globally uniformly asymptotically stable (and $\NRM{X_t}$ goes to 0 exponentially 
fast as $t\to\infty$ for any function $u$) if $R(\alpha^2)<1$, 
\item uniformly stable (but for some functions $u$, $\NRM{X_t}$ does not converge 
to 0) if $R(\alpha^2)=1$, 
\item unbounded if $R(\alpha^2)>1$,
\end{itemize}
where $R(\alpha^2)$ is given by \eqref{eq:R}.
\end{itemize}

\begin{proof}[Proof of Theorem~\ref{th:BBM}]
The general case is 
considered in \cite{1}. The main idea is to construct the so-called \emph{worst 
trajectory} choosing at each instant of time the vector field that drives the particle
away from the origin. The solutions $x_t=(y_t,z_t)$ of $\dot x_t=A_0x_t$ and 
$\dot x_t=A_1x_t$ starting from $x_0=(y_0,z_0)$ are respectively given by 
\[
\begin{cases}
 y_t=(z_0t+y_0) e^{-\alpha t}&\\
 z_t=z_0 e^{-\alpha t} 
\end{cases}
\quad\text{and}\quad
\begin{cases}
 y_t=y_0 e^{-\alpha t}& \\
 z_t=(-y_0t+z_0) e^{-\alpha t}.
\end{cases}
\]
Let us define, for $x=(y,z)$, 
\[
Q(x)=\det(A_0x,A_1x)=\alpha y^2-yz-\alpha z^2.
\]
Then the set of the points where $A_0x$ and $A_1x$ are collinear is given by
\[
\BRA{x\in\dR^2\ :\ Q(x)=0}=\BRA{x=(y,z)\ :\ y=\gamma^+z\text{ or }y=\gamma^-z}
\]
where 
\[
\gamma^+=\frac{1+\sqrt{1+4\alpha^2}}{2\alpha}>0
\quad\text{and}\quad
\gamma^-=\frac{1-\sqrt{1+4\alpha^2}}{2\alpha}<0.
\]
Let us start with $x_0=(0,1)$ and $I_0=0$. Choose $t_1=\gamma^+$ in such a way 
that: 
\[
x_{t_1}=\PAR{\gamma^+e^{-\alpha\gamma^+},e^{-\alpha\gamma^+}}.
\] 
Now, set $t_2=t_1+\gamma^+-\gamma^-$ and $I_t=1$ for $t\in[t_1,t_2)$ 
in such a way that $y_{t_2}=\gamma^- z_{t_2}$ \emph{i.e.} 
$y_{t_2}=-(\gamma^+)^{-1}z_{t_2}$. Then, one has
\[
x_{t_2}=\PAR{\gamma^+e^{-\alpha(2\gamma^+-\gamma^-)},-(\gamma^+)^2
e^{-\alpha(2\gamma^+-\gamma^-)}}.
\]
Finally, choose $t_3=t_2-\gamma^-$ and $I_t=0$ for $t\in[t_2,t_3)$ in such a way that 
$y_{t_3}=0$. Then, one has
\[
x_{t_3}=\PAR{0,-(\gamma^+)^2e^{-2\alpha(\gamma^+-\gamma^-)}}.
\]
The process is unbounded if and only if $\NRM{x_{t_3}}>1$. This is equivalent to \eqref{eq:R}.
\end{proof}


\begin{figure}
\begin{center}
 \includegraphics[scale=0.25]{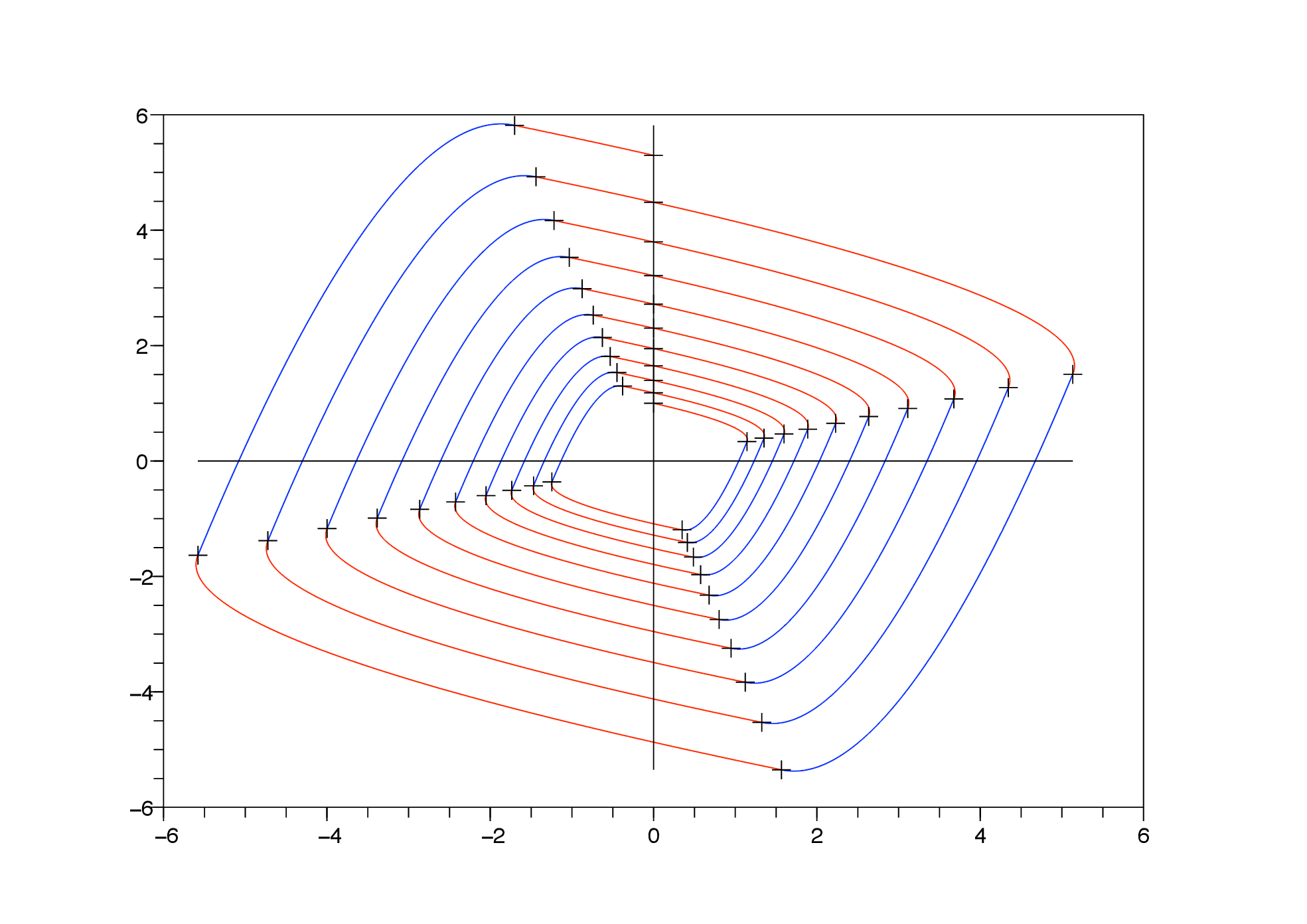}
 \includegraphics[scale=0.25]{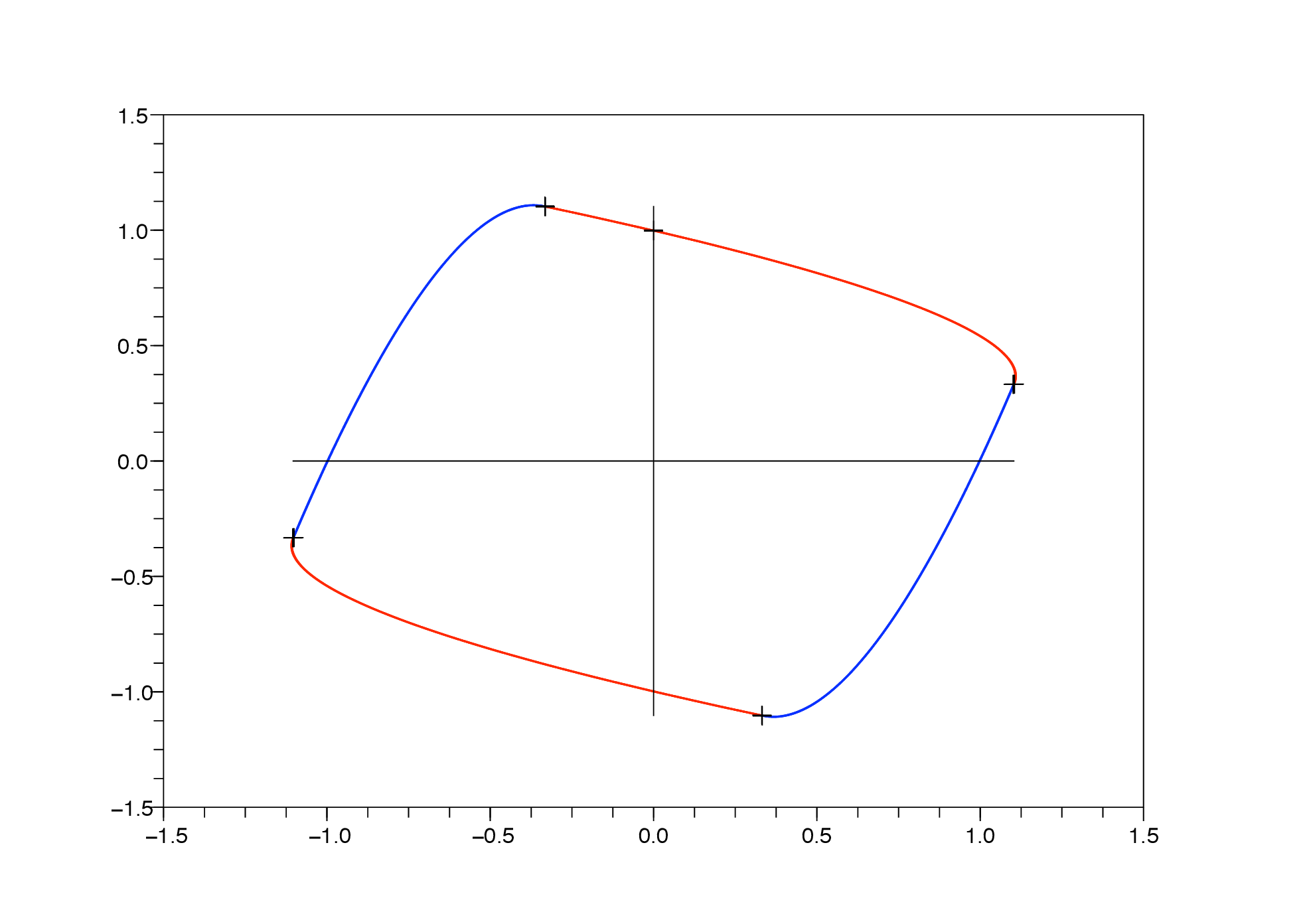}
 \includegraphics[scale=0.25]{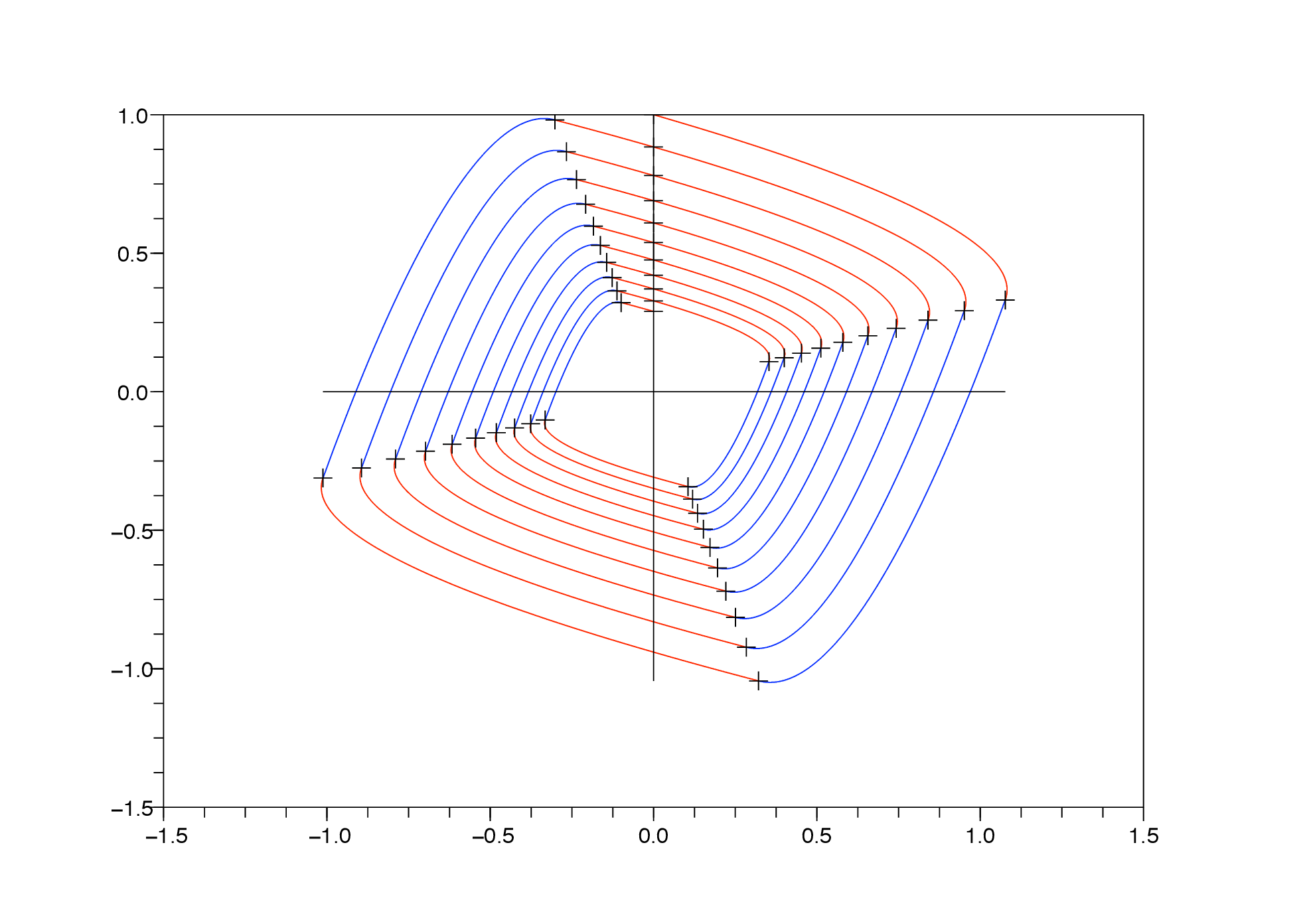}
 \caption{The \emph{worth trajectory} with $\alpha=0.32$ (on the left), 
 $\alpha=0.3314$ (in the middle) and $\alpha=0.34$ (on the right). 
 The system evolves clock-wisely from $(0,1)$.}
 \label{fi:main}
\end{center}
\end{figure}

\subsection{Invariant measure(s) of switched flows}

In order to avoid the possible explosions studied in Section~\ref{sec:blowup}, one can 
impose that the state space of the continuous variable is a compact set. 

In \cite{BLMZ3}, it is shown thanks to an example that the number of the invariant 
measures may depend on the jump rate for fixed vector fields (as for the problem of 
(un)-stability described in the previous section).  Moreover Hörmander-like conditions 
on the vector fields are formulated in \cite{BH,BLMZ3} to ensure that the first marginal 
of the invariant measure(s) may be absolutely continuous with respect to the Lebesgue 
measure on $\dR^d$. However the density may blow up as it is shown in the example 
below. 
 
\begin{xpl}[Possible blow up of the density near a critical point]\label{ex:dim1}
 Consider the process on $\dR\times\BRA{0,1}$ associated to the infinitesimal generator 
\[
 Lf(x,i)=-\alpha_i(x-i)\partial_x f(x,i)+\lambda_i (f(x,1-i)-f(x)).
\]
This process is studied in \cite{KB,RMC}. The support of its invariant measure $\mu$ is the 
set $[0,1]\times\BRA{0,1}$ and $\mu$ is given by 
\[
\int f\,d\mu=\frac{\lambda_1}{\lambda_0+\lambda_1}\int_0^1\! f(x,0)p_0(x)\,dx
+\frac{\lambda_0}{\lambda_0+\lambda_1}\int_0^1\! f(x,1)p_1(x)\,dx, 
\]
where $p_0$ and $p_1$ are Beta distributions: 
\[
p_0(x)= \frac{x^{\lambda_0/\alpha_0-1}(1-x)^{\lambda_1/\alpha_1} }
{B(\lambda_0/\alpha_0,\lambda_1/\alpha_1+1)} 
\quad\text{and}\quad
p_1(x)= \frac{ x^{\lambda_0/\alpha_0}(1-x)^{\lambda_1/\alpha_1-1} }
{B(\lambda_0/\alpha_0+1,\lambda_1/\alpha_1)}.
\]
The density of the invariant measure possibly explodes near 0 or 1.
\end{xpl}
The paper \cite{BHM} is a detailed analysis of invariant measures for switched 
flows in dimension one. In particular, the authors prove smoothness of the 
invariant densities away from critical points and describe the asymptotics of 
the invariant densities at critical points. 

The situation is more intricate for higher dimensions. 
\begin{xpl}[Possible blow up of the density in the interior of the support]\label{ex:dim2}
Consider the 
process on $\dR^2\times\BRA{0,1}$ associated to the constant jump rates $\lambda_0$ 
and $\lambda_1$ for the discrete component and the vector fields 
\begin{equation}\label{eq:drotation}
F^0(x)=Ax \text{ and }F^1(x)=A(x-a) 
\text{ where } 
A=
\begin{pmatrix}
-1 & -1 \\
1 & -1 
\end{pmatrix}
\quad\text{and}\quad 
a=
\begin{pmatrix}
 1\\
 0
\end{pmatrix}.
\end{equation}
\begin{figure}
\begin{center}%
  \includegraphics[height=8cm,width=15cm]{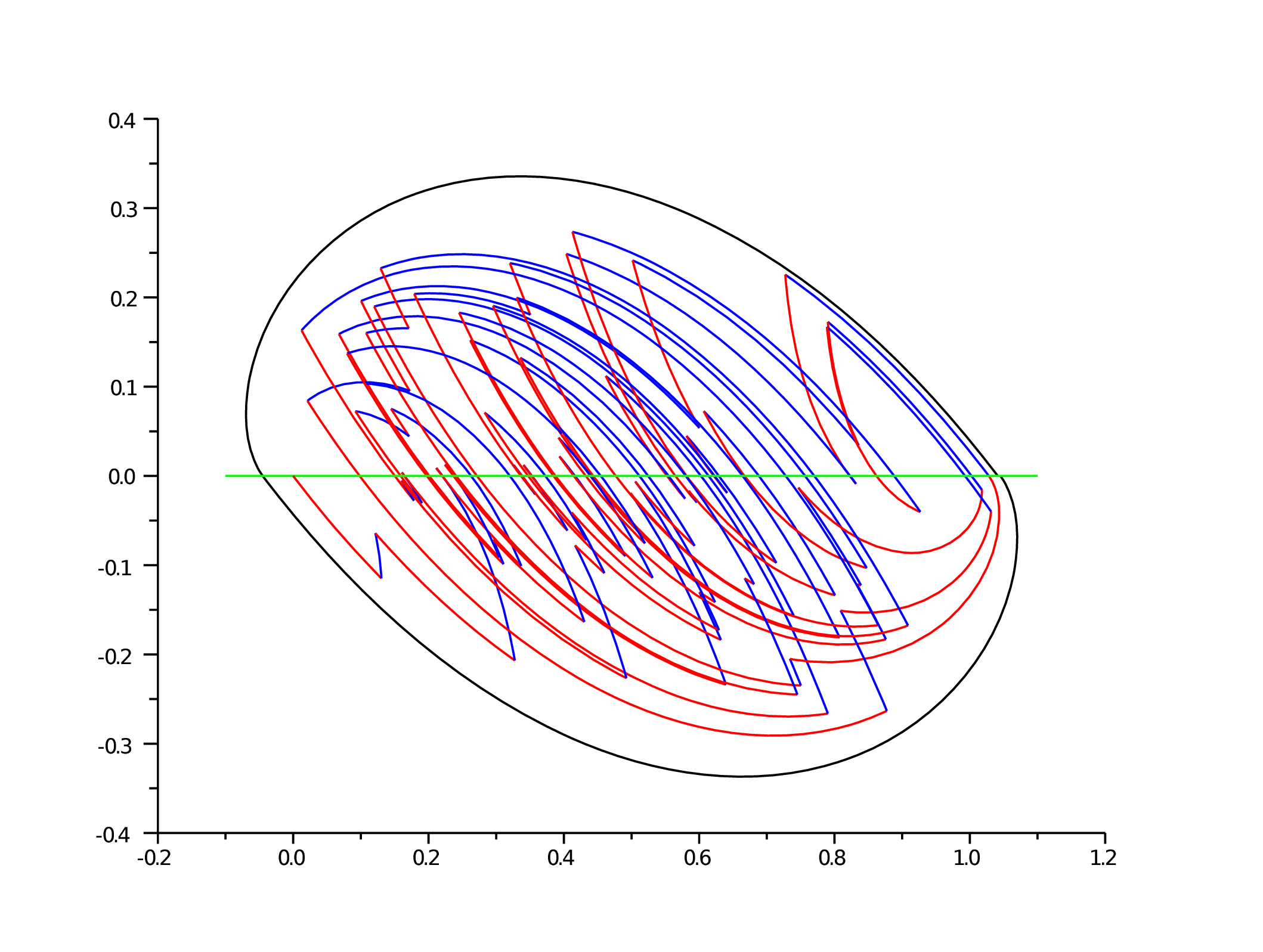}%
\caption{Path of the process associated to $F^0$ and $F^1$ given by \eqref{eq:drotation} 
starting from the origin. Red (resp. blue) pieces of path correspond to $I=1$ (resp. $I=0$).}
\label{fi:drotation}
\end{center}%
\end{figure}
The origin and $a$ are the respective unique critical points of $F^0$ and $F^1$. Thanks to 
the precise estimates in~\cite{BHM}, one can prove the following fact. If 
$\lambda_0$ is small enough then, as for one-dimensional example, the density of the 
invariant measure blows up at the origin. This also implies that the density is infinite 
on the set $\BRA{\varphi^1_t(0)\ : \ t\geq 0}$. 
\end{xpl}

\begin{ques}
 What can be said on the smoothness of the density of the invariant measure of such 
 processes?  
\end{ques}

\subsection{A convergence result}

This section sums up the study of the long time behavior of certain switched flows 
presented in~\cite{BLMZ2}. See also \cite{shao} for another approach. 
To focus on the main lines of this paper, the hypotheses below are far from the optimal ones. 

\begin{hyp}[Regularity of the jump rates]\label{ass:lambda}
There exist $\underline a>0$ and $\kappa>0$ 
such that, for any $x,\tilde x\in\dR^d$ and $i,j\in E$, 
\[
a(x,i,j)\geq \underline a
\quad\text{and}\quad
\sum_{j\in E}\ABS{a(x,i,j)-a(\tilde x,i,j)}\leq \kappa\NRM{x-\tilde x}.
\]
\end{hyp}
The lower bound condition insures that the second --- discrete ---  coordinate of $Z$ changes 
often enough (so that the second coordinates of two independent copies of $Z$ coincide 
sufficiently often). 
\begin{hyp}[Strong dissipativity of the vector fields]\label{as:dissip}
There exists  $\alpha>0$ such that, 
\begin{equation}\label{eq:dissip}
\DP{x-\tilde x,F^i(x)-F^i(\tilde x)}\leq -\alpha\NRM{x-\tilde x}^2,
\quad
x,\tilde x\in\dR^d,\ i\in E.
\end{equation}
\end{hyp}
Hypothesis~\ref{as:dissip} ensures that, for any $i\in E$, 
\[
\NRM{\varphi^i_t(x)-\varphi^i_t(\tilde x)}\leq  e^{-\alpha t}\NRM{x-\tilde x},
\quad
x,\tilde x\in \dR^d.
\]  
As a consequence, the vector fields $F^i$ has exactly one critical point 
 $\sigma(i)\in\dR^d$. Moreover it is exponentially stable since, for any 
 $x\in\dR^d$, 
\[
  \NRM{\varphi^i_t(x)-\sigma(i)}\leq e^{-\alpha t}\NRM{x-\sigma(i)}.
\]
In particular, $X$ cannot escape from a sufficiently large ball $\bar B(0,M)$.
Define the following distance $\cW_1$ on the probability measures on 
$B(0,M)\times E$: for $\eta,\tilde\eta\in\cP(B(0,M)\times E)$,   
\[
\cW_1(\eta,\tilde\eta)=\inf\BRA{\dE\vert X-\tilde X\vert+\dP(I\neq \tilde I)
\,:\, (X,I)\sim \eta\text{ and }(\tilde X,\tilde I)\sim \tilde \eta}.
\] 

\begin{thm}[Long time behavior \cite{BLMZ2}]\label{th:nonconstant}
Assume that Hypotheses \ref{ass:lambda} and \ref{as:dissip} hold. 

Then, the process has a unique invariant measure and its support is 
included in $\bar B(0,M)\times E$. Moreover, 
let $\nu_0$ and $\tilde \nu_0$ be two probability measures on $\bar B(0,M)\times E$. 
Denote by $\nu_t$ the law of $Z_t$ when $Z_0$ is distributed as $\nu_0$ 
Then there exist positive constants $c$ and $\gamma$ such that
\[
\cW_1\PAR{\eta_t,\tilde \eta_t}\leq %
c e^{-\gamma t}.
\]
\end{thm}
The constants $c$ and $\gamma$ can be explicitly expressed in term of the 
parameters of the model (see \cite{BLMZ2}). The proof relies on the construction 
of an explicit coupling. See also \cite{cloez-hairer,monmarche2}.

\begin{ques}
One can apply Theorem~\ref{th:nonconstant} to the processes defined in 
Examples~\ref{ex:dim1} and \ref{ex:dim2}. The associated time reversal processes 
are associated to unstable vector fields and unbounded jump rates. What can be said 
about their convergence to equilibrium? 
\end{ques}

Section~\ref{sec:neuron} present an application of this theorem to a biological 
model. In Section~\ref{sec:chemo}, we describe a naïve model for the movement 
of bacteria that can also be seen as an ergodic telegraph process.

\subsection{Neuron activity} \label{sec:neuron}

The paper \cite{PTW} establishes limit theorems for a class of stochastic
hybrid systems (continuous deterministic dynamic coupled with jump Markov
processes) in the fluid limit (small jumps at high frequency), thus extending
known results for jump Markov processes. The main results are a functional law
of large numbers with exponential convergence speed, a diffusion
approximation, and a functional central limit theorem. These results are then
applied to neuron models with stochastic ion channels, as the number of
channels goes to infinity, estimating the convergence to the deterministic
model. In terms of neural coding, the central limit theorems allows to
estimate numerically the impact of channel noise both on frequency and spike
timing coding.

The Morris--Lecar model introduced in \cite{ML81} describes the evolution in time of the electric 
potential $V(t)$ in a neuron. The neuron exchanges different ions with its environment via ion 
channels which may be open or closed. In the original --~deterministic~-- model, the proportion 
of open channels of different types are coded by two functions $m(t)$ and $n(t)$, and the 
three quantities $m$, $n$ and $V$ evolve through the flow of an ordinary differential equation. 

Various stochastic versions of this model exist. Here we focus on a model 
described in~\cite{WTP10}, to which we refer for additional information. 
This model is motivated by the fact that
 $m$ and $n$, being proportions of open channels, are better 
coded as discrete variables. More precisely, we fix a large integer $K$ (the total 
number of channels) and define a PDMP $(V,u_1,u_2)$ with values in 
$\dR \times \{0,1/K,2/K \dots, 1\}^2$ as follows. 

Firstly, the potential $V$ evolves according to
\begin{equation}
  \frac{dV(t)}{dt} = \frac{1}{C} \left(I - \sum_{i=1}^3 g_iu_i(t) (V - V_i) \right)
  \label{eq=evolV}
\end{equation}
where $C$ and $I$ are positive constants (the capacitance and input current), the $g_i$ and $V_i$ 
are positive constants (representing conductances and equilibrium potentials for different types 
of ions), $u_3(t)$ is equal to $1$ and $u_1(t)$, $u_2(t)$ are the (discrete) proportions of open 
channels for two types of ions. 

These two discrete variables follow birth-death processes on $\{0, 1/K, \ldots,1\}$ with 
birth rates $\alpha_1$, $\alpha_2$ and death rates $\beta_1$, $\beta_2$  that depend on the 
potential $V$: 
\begin{equation}
  \begin{aligned}
  \alpha_i(V) &= c_i \cosh\left( \frac{V - V'_i}{2V''_i}\right) \left(1 + \tanh\left(\frac{V - V'_i}{V''_i}\right)
  \right) \\
  \beta_i(V) &= c_i \cosh\left( \frac{V - V'_i}{2V''_i}\right) \left(1 - \tanh\left(\frac{V - V'_i}{V''_i}\right)
  \right)
  \end{aligned}
  \label{eq=rates}
\end{equation}
where the $c_i$ and $V'_i$, $V''_i$ are constants. 

Let us check  that Theorem~\ref{th:nonconstant} can be applied in this example. 
Formally the process is a PDMP with $d=1$ and the finite set $E = \{0, 1/K,\ldots,1\}^2$. 
The discrete process $(u_1,u_2)$ plays the role of the index $i\in E$, and the fields
$F^{(u_1,u_2)}$ are defined (on $\dR$) by \eqref{eq=evolV} by setting $u_1(t) = u_1$, 
$u_2(t) = u_2$. 

The constant term $u_3g_3$ in \eqref{eq=evolV} ensures that the uniform dissipation property
\eqref{eq:dissip} is satisfied: for all $(u_1,u_2)$, 
\begin{align*}
  \DP{ V - \tilde{V}, F^{(u_1,u_2)}(V) - F^{(u_1,u_2)}(\tilde{V})}
  &= - \frac{1}{C}\sum_{i=1}^3 u_ig_i(V-\tilde{V})^2 \\
  &\leq - \frac{1}{C}u_3g_3 (V-\tilde{V})^2.
\end{align*}

The Lipschitz character and the bound from below on the rates are not immediate. 
Indeed the jump rates \eqref{eq=rates} are not bounded from below if $V$ is allowed to 
take values in $\dR$. 

However, a direct analysis of \eqref{eq=evolV} shows that $V$ is essentially bounded~: 
all the fields $F^{(u_1,u_2)}$ point inward at the boundary of the (fixed) line segment
$\mathcal{S} = [0,\max(V_1,V_2, V_3 + (I+1)/g_3u_3)]$, so if $V(t)$ starts
in this region it cannot get out. The necessary bounds all follow by compactness, since
$\alpha_i(V)$ and $\beta_i(V)$ are $\mathcal{C}^1$ in $\mathcal{S}$ and  strictly positive.

\subsection{Chemotaxis}\label{sec:chemo}

Let us briefly describe how bacteria move (see 
\cite{MR949094,MR2122993,MR2123062} for details). 
They alternate two basic behavioral modes: a more or less linear
motion, called a run, and a highly erratic motion, called tumbling, the
purpose of which is to reorient the cell. During a run the bacteria move at
approximately constant speed in the most recently chosen direction. Run times
are typically much longer than the time spent tumbling. In practice, the
tumbling time is neglected. An appropriate stochastic process for describing
the motion of cells is called the \emph{velocity jump process} which is deeply
studied in \cite{MR949094}. The velocity belongs to a compact set (the unit
sphere for example) and changes by random jumps at random instants 
of time. Then, the position is deduced by integration of the velocity. The jump 
rates may depend on the position when the medium is not homogeneous: 
when bacteria move in a favorable direction \emph{i.e.} either in the direction 
of foodstuffs or away from harmful substances the run times are 
increased further. Sometimes, a diffusive approximation is 
available \cite{MR949094,rousset}.

In the one-dimensional simple model studied in \cite{FGM}, the particle evolves in
$\dR$ and its velocity belongs to $\BRA{-1,+1}$. Its infinitesimal generator
is given by:
\begin{equation}\label{eq:giabs}
Af(x,v)=v\partial_xf(x,v)+\PAR{a+(b-a)\ind_\BRA{xv>0}}(f(x,-v)-f(x,v)),
\end{equation}
with $0<a<b$. The dynamics of the process is simple: when $X$ goes aways 
from 0, (resp. goes to 0), $V$ flips to $-V$ with rate $b$ (resp. $a)$. Since $b>a$, 
it is quite intuitive that this Markov process is ergodic. One could think about it 
as an analogue of the diffusion process solution of 
 \[
dZ_t=dB_t-\text{sign}(Z_t)\,dt.
\]
More precisely, under a suitable scaling, one can show that $X$ goes to $Z$. 
Finally, this process is an ergodic version of the so-called telegraph process. 
See for example~\cite{MR0510166,MR2652885}.
 
Of course, this process does not satisfy the hypotheses of Theorem~\ref{th:nonconstant}
since the vector fields have no stable point. It is shown in \cite{FGM} that the invariant
measure $\mu$ of $(X,V)$ driven by \eqref{eq:giabs} is the product measure on
$\dR_+\times\BRA{-1,+1}$ given by
\[
\mu(dx,dv)=(b-a)e^{-(b-a)x}\,dx\otimes  \frac{1}{2}(\delta_{-1}+\delta_{+1})(dv).
\]
One can also construct an explicit coupling to get explicit bounds for the
convergence to the invariant measure in total variation norm \cite{FGM}. See also 
\cite{Monmarche} for another approach, linked with functional inequalities. 
 
\begin{ques}[More realistic models]
Is it possible to establish quantitative estimates for the convergence to equilibrium 
for more realistic dynamics (especially in $\dR^3$) as considered in 
\cite{MR949094,MR2122993,MR2123062}?
\end{ques}

\paragraph*{Acknowledgements.}
FM deeply thanks Persi Diaconis for his energy, curiosity and enthusiasm and Laurent Miclo 
for the perfect organisation of the stimulating workshop "Talking Across Fields" in Toulouse during 
March 2014. This paper has been improved thanks to the constructive comments of two referees. 
FM acknowledges financial support from the French ANR 
project ANR-12-JS01-0006 - PIECE.

\addcontentsline{toc}{section}{\refname}%
{ \footnotesize
\bibliography{taf-malrieu}
\bibliographystyle{amsplain}
}

{\footnotesize %
 \noindent Florent \textsc{Malrieu},
 e-mail: \texttt{florent.malrieu(AT)univ-tours.fr}

 \medskip

 \noindent\textsc{Laboratoire de Mathématiques et Physique Théorique (UMR
CNRS 6083), Fédération Denis Poisson (FR CNRS
2964), Université François-Rabelais, Parc de Grandmont,
37200 Tours, France.}

}

\end{document}